\documentclass[11pt]{amsart}
\usepackage{amsmath,amssymb,amsthm,upref,graphicx,mathrsfs, enumerate,citeref}
\usepackage{esint}
\usepackage{textcomp}
\usepackage{array,comment}
\usepackage{color,xcolor}
\usepackage{latexsym}
\usepackage{amsmath}
\usepackage{amssymb}
\usepackage{amsthm} 
\usepackage{amscd}
\usepackage{color}
\usepackage{comment}

\numberwithin{equation}{section}
\allowdisplaybreaks

\newtheorem{theorem}{Theorem}[section]

\newtheorem{corollary}[theorem]{Corollary}

\theoremstyle{definition}
\newtheorem{definition}[theorem]{Definition}
\newtheorem{example}[theorem]{Example}
\newtheorem{remark}[theorem]{Remark}

\begin{document}

\title[Wavelet characterization and extension operators]{Applications of extrapolations 
to wavelet characterization
of various function spaces and extension operators}

\author{Mitsuo Izuki, Takahiro Noi and Yoshihiro Sawano}
\address{
Faculty of Liberal Arts and Sciences\\
Tokyo City University\\
1-28-1 Tamazutsumi, Setagaya-ku, Tokyo 158-8557, Japan
}
\email{izuki@tcu.ac.jp}

\address{
Department of Mathematical and Data Science\\
Otemon Gakuin University\\
2-1-15 Nishiai, Ibaraki, Osaka 567-8502, Japan
}
\email{taka.noi.hiro@gmail.com}

\address{
Department of Mathematics\\
Graduate School of Science and Engineering\\
Chuo University\\
1-13-27 Kasuga, Bunkyo-ku, Tokyo, Japan
}
\email{yoshihiro-sawano@celery.ocn.ne.jp}
\maketitle

\begin{abstract}
The aim of this paper is to apply an extrapolation result without relying on convexification. 
We characterize ball Banach function spaces in terms of wavelets, 
formulated in a way that takes into account the smoothness properties of the spaces under consideration. 
The same technique can also be applied to prove vector-valued inequalities, for example. 
Furthermore, the result presented here refines a recent extension operator result by Zhu, Yang, and Yuan.
\end{abstract}

{\bf 2020 Classification} 42B35, 41A17, 26B33


\section{Introduction}

Wavelets were originally defined as functions that form an orthonormal basis in the $L^2$ space, and their mathematical theory has seen significant development. Nowadays, various types of wavelets are known, possessing desirable properties such as smoothness, rapid decay, compact support, and band-limitedness. By leveraging these properties appropriately, it is possible to characterize function spaces via alternative norms and to construct bases with excellent properties across a wide range of spaces beyond just $L^2$.

The wavelet characterization of function spaces is a fascinating subject in real analysis and applied mathematics. In particular, the interplay between weights and wavelets has been extensively developed. Such characterizations are especially useful when analyzing the boundedness of operators. Here and below,
by a ``weight", we mean a non-negative measurable function that is positive almost everywhere.

In this paper, we explore the connection between wavelets and an extrapolation result for ball Banach function spaces. Our assumptions are stated entirely in terms of the ball Banach function space and its K\"{o}the dual. As an application, we obtain wavelet characterizations of various function spaces.

As mentioned above, wavelet characterizations of function spaces can be effectively used to investigate the boundedness properties of operators. In recent years, the use of weights has become extremely useful in the analysis of operators arising in harmonic analysis. 

The study of weighted function spaces has a rich history, grounded in Muckenhoupt's theory in real analysis. For an overview, see \cite[Introduction]{ABM}. In particular, weighted Lebesgue spaces have been characterized via the boundedness of integral operators; see, for example, \cite{GM2001, Le}. This method extends to more general function spaces, including weighted Lebesgue spaces with variable exponent \cite{INS-2015} and Herz-type spaces \cite{IzukiAnal2010, Tachi}.

Furthermore, the wavelet characterization of Herz-type spaces has been achieved via the $\varphi$-transform \cite{HWY}. Another approach to the wavelet characterization of function spaces involves the use of extrapolation theorems. This method has been applied, for instance, in the setting of Lebesgue spaces with variable exponent \cite{IzukiGMJ, Kop}.

We recall the definition of ball Banach function spaces defined in \cite{shyy17}.
The space $L^0({\mathbb R}^n)$ denotes the space of all equivalence classes of 
measurable functions modulo null functions.
Let $\rho$ be a mapping which maps $L^0({\mathbb R}^n)$ to $[0,\infty]$.
Recall that it is called a {\it ball Banach function norm} (over ${\mathbb R}^n$) 
if, for all non-negative measurable functions $f, g, f_j$ $(j=1, 2, 3, \ldots)$, 
for all constants $a \ge 0$ and for all open balls $B$ in ${\mathbb R}^n$, 
the following properties hold:
\begin{enumerate}
\item[$({\rm P}1)$]
$\rho(f)=0\,\Leftrightarrow\,f=0$ a.e.;
$\rho(af)=a\rho(f)$,
$\rho(f+g)\le\rho(f)+\rho(g)$,
\item[$({\rm P}2)$]
$\rho(g)\le\rho(f)$
if
$0\le g\le f\text{ a.e.}$,
\item[$({\rm P}3)$]
the Fatou property;
$\rho(f_{j})\uparrow\rho(f)$
holds
whenever
$0\le f_{j}\uparrow f\text{ a.e.}$,
\index{Fatou property@Fatou property}
\item[$({\rm P}4)$]
$\rho(\chi_B)<\infty$,
\item[$({\rm P}5)$]
$\displaystyle
\|\chi_B f\|_{L^{1}} \lesssim\rho(f)$
with the implicit constant
depending on $B$ and $\rho$ but independent of $f$.
\end{enumerate}
\index{ball Banach function norm@ball Banach function norm}
\index{ball Banach function space@ball Banach function space}
Accordingly, the space generated by such $\rho$,
which is given by
\[
X=X_\rho=\{f \in L^0({\mathbb R}^n)\,:\,f
\mbox{
is a measurable function satisfying }
\|f\|_{X_\rho}:=\rho(|f|)<\infty\},
\] is called a ball
Banach function space.
Furthermore its \lq\lq associate norm"
$\rho'$ is defined for a non-negative measurable function $g$ by
$
\rho'(g)
:=
\sup\left\{
\|f \cdot g\|_{L^{1}}\,:\,
f\in X,
\rho(f)\le 1\right\}.
$
Likewise we can consider the K\"{o}the dual
$X'=(X_\rho)'=X_{\rho'}$
for Banach lattices.

We will use the Hardy--Littlewood maximal operator
$M$ to describe our assumption.
For $f \in L^0({\mathbb R}^n)$,
one
defines a function $M f$ by
$$\displaystyle
M f(x):=
\sup\limits_{r>0}\frac{1}{r^n}\int_{|y-x|<r} |f(y)|{\rm d}y
$$
for $x \in {\mathbb R}^n$.
The mapping $f \mapsto M f$ is called the Hardy--Littlewood maximal operator.

In \cite{Rutsky} Rutsky characterized the condition
for which the Riesz transforms, given by
\[
{\mathcal R}_{j} f(x):=
\lim\limits_{\varepsilon \downarrow 0}\int_{|y-x|>\varepsilon}
\frac{x_{j}-y_{j}}{|x-y|^{n+1}}f(y){\rm d}y
\quad (x \in {\mathbb R}^n),
\]
for $j=1,2,\ldots,n$,
are bounded on $X$.
In fact, he showed that $M$ is bounded on $X$ and on $X'$
if and only if ${\mathcal R}_j$ is bounded on $X$
for all $j=1,2,\ldots,n$.

As we mentioned, weights are fundamental tools to investigate
the boundedness of operators. Let $w$ be a weight
and $1<p<\infty$.
We also consider $L^p(w)$ which collects all
$f \in L^0({\mathbb R}^n)$
for which $\|f\|_{L^p(w)}:=\|f w^{\frac1p}\|_{L^p}$ is finite. 
Using the Muckenhoupt class $A_p$, which will be recalled in (\ref{eq:241226-111}) below, we begin by recalling an extrapolation theorem. We then extend this theorem to the case of the local Muckenhoupt class $A_{p,\mathrm{loc}}$.

Our main focus in this paper is on applications of these extrapolation results. This approach clearly demonstrates the fundamental role that weights play in analyzing the boundedness properties of operators.

\begin{theorem}\label{thm:241102-1a}
Let $X$ be a ball Banach function space and $1<p<\infty$.
Assume that
$M$ is bounded on $X$ and on $X'$.
For a given increasing function
$N:[1,\infty) \to (0,\infty)$,
write
\begin{equation}\label{eq:241222-1}
{\mathcal F}
:=
\{(f,g)\in L^0({\mathbb R}^n)^2\,:\,\|f\|_{L^p(w)} \le N([w]_{A_p})\|g\|_{L^p(w)}
\mbox{ for all }w \in A_p\}.
\end{equation}
Then
there exists a constant $C>0$ such that
$\|f\|_X \le C\|g\|_X$
for all $(f,g) \in {\mathcal F}$.
\end{theorem}
The proof of Theorem~\ref{thm:241102-1a} is known; see
\cite[Theorem 10.1]{C17},
\cite[Theorem 3.1]{CMM22}
and
\cite[Theorem~A, Theorem 4.7 and Remark 4.8]{Zoe}.
This theorem is a special case of~\cite[Theorem~A]{Zoe} with
$X = Y$, $r_1 = r_2 = 1$, and $s_1 = s_2 = \infty$.
For clarity and convenience, we restate Theorem~\ref{thm:241102-1a}
in a more accessible form below.

\begin{corollary}\label{cor:241222-1}
Let $X$ be a Banach function space and $1<p<\infty$.
Then the following are equivalent{\rm:}
\begin{enumerate}
\item[$(A)$]
$M$ is bounded on $X$ and on $X'$.
\item[$(B)$]
The Riesz transform ${\mathcal R}_j$ is bounded on $X$
for each $j=1,2,\ldots,n$.
\item[$(C)$]
Define ${\mathcal F}$ by $(\ref{eq:241222-1})$.
Then
there exists a constant $C>0$ such that
$\|f\|_X \le C\|g\|_X$
for all $(f,g) \in {\mathcal F}$.
\end{enumerate}
\end{corollary}
Corollary~\ref{cor:241222-1} is also a known result. 
See \cite{Nieraeth} for more equivalent conditions.
Nevertheless, we briefly recall its proof for completeness.

Corollary~\ref{cor:241222-1} follows directly from a combination of a known equivalence and Theorem~\ref{thm:241102-1a}. 
As previously mentioned, Rutsky established the equivalence between conditions~$(A)$ and~$(B)$. 
Theorem~\ref{thm:241102-1a} asserts that~$(A)$ implies~$(C)$. 
Finally, it is known that $({\mathcal R}_j f, f) \in {\mathcal F}$ for $N(t) = \alpha t^{\max(p,p')}$ with sufficiently large $\alpha$; see~\cite[Theorem~305]{book},
for example. 
This yields the implication $(C) \Rightarrow (B)$, completing the proof.

We have an analogy to local  Hardy--Littlewood maximal operator
$M_{\rm loc}$.
For $f \in L^0({\mathbb R}^n)$,
one
defines a function $M_{\rm loc} f$ by
$$\displaystyle
M_{\rm loc} f(x):=
\sup\limits_{0<r<1}\frac{1}{r^n}\int_{|y-x|<r} |f(y)|{\rm d}y
$$
for $x \in {\mathbb R}^n$.
The mapping $f \mapsto M_{\rm loc} f$ is called the local Hardy--Littlewood maximal operator.
\begin{theorem}\label{thm:241102-1ab}
Let $X$ be a ball Banach function space and $1<p<\infty$.
Assume that
$M_{\rm loc}$ is bounded on $X$ and on $X'$.
For a given increasing function
$N:[1,\infty) \to (0,\infty)$,
write
\begin{equation}\label{eq:241222-1k}
{\mathcal F}_{\rm loc}
:=
\{(f,g)\in L^0({\mathbb R}^n)^2\,:\,\|f\|_{L^p(w)} \le N([w]_{A_{p,\rm loc}})\|g\|_{L^p(w)}
\mbox{ for all }w \in A_{p,\rm loc}\}.
\end{equation}
Then
there exists a constant $C>0$ such that
$\|f\|_X \le C\|g\|_X$
for all $(f,g) \in {\mathcal F}_{\rm loc}$.
\end{theorem}
This result is also known:
see \cite[Theorem 3.1]{CMM22}.
For readers' convinience,
we provide a short proof of Theorem \ref{thm:241102-1ab}
as an appendix in Section \ref{s2}.
Here, we make a clarifying remark on the strength of
the assumption in Theorems \ref{thm:241102-1a}
and \ref{thm:241102-1ab} regarding
$X$.
\begin{remark}\label{rem:241226-1} If $X$ satisfies the assumption in Theorem
\ref{thm:241102-1a}, meaning that $M$ is bounded on both $X$
and $X'$, then according to Lerner and P\'{e}rez
\cite{Lerner07}, there exists a constant $\eta > 1$ such
that the operator $f \mapsto (M[|f|^\eta])^{\frac{1}{\eta}}$
is bounded on both $X$ and $X'$. In this scenario, we can
establish $X \hookrightarrow L^\eta(w)$ for some $w \in A_1$.
We refer
to \cite[Lemma 2.7]{MaSa25}.
See also \cite[Lemma 3]{Sawano2024}. 
An analogy to $M_{\rm loc}$ is available to
Theorem \ref{thm:241102-1ab}.
\end{remark}

Curbera, Garc\'{i}a-Cuerva, Martell, and P\'{e}rez made a 
significant breakthrough by obtaining the extrapolation 
theorem for rearrangement invariant Banach function spaces 
\cite{CGMP06}. The case where $X$ is not rearrangement 
invariant has gained increasing attention. For instance, in 
\cite[Theorem 3.3]{ho12} and \cite[Theorem 3.1]{DSL23}, Ho employed $p$-convexification to 
achieve a vector-valued extension of the maximal inequality. 
For a ball Banach function space $X$ and $p>0$, we consider 
its {\it $p$-convexification given by} 
$X^p:=\{f \in L^0({\mathbb R}^n)\,:\,|f|^{\frac1p} \in X\}$. 
The norm of $X^p$ is given by 
$\|f\|_{X^p}:=(\|\,|f|^{\frac1p}\,\|_{X})^p$. Remarkably, 
Theorems \ref{thm:241102-1a}
and \ref{thm:241102-1ab} do not require $p$-convexification.
Nor do we need the absolute continuity of the norm
of $X$.
Recall that $X$ is said to have the absolutely continuous norm
if $\|f\chi_{E_j}\|_X \to 0$ as $j \to \infty$
for any $f \in X$
and
any sequence $\{E_j\}_{j=1}^\infty$ of sets decreasing to $\emptyset$.
If the space $X$ is separable,
then $X$ has the absolutely continuous norm.
See \cite{BS}.

Here and below,
in addition to the above notation, 
we use the following convention in this paper.
\begin{itemize}
\item
By a \lq \lq cube" we mean a compact cube
whose edges are parallel to the {\it coordinate axes}.
The metric closed ball defined 
by $\ell^\infty$ is called a {\it cube}.
If a cube has center $x$ and radius $r$,
we denote it by $Q(x,r)$.
Namely, we write
\[
Q(x,r) \equiv
\left\{y=(y_{1},y_{2},\ldots,y_n) \in {\mathbb R}^n\,:\,
\max\limits_{j=1,2,\ldots,n}|x_{j}-y_{j}| \le r\right\}
\]
when $x=(x_{1},x_{2},\ldots,x_n) \in {\mathbb R}^n$
and $r>0$.
The symbol ${\mathcal Q}$ stands for all cubes.
\item
Given $x\in \mathbb{R}^n$ and $r>0$, we denote the
open ball centered at $x$ and having radius $r>0$ by
$B(x,r)$.
We abbreviate
$B(r):=B(0,r)$.
\item
Let $A,B \ge 0$.
Then $A \lesssim B$ means
that there exists a constant $C>0$
such that $A \le C B$,
where $C$ is usually independent of the functions we are considering.
The symbol $A \sim B$ stands for the two-sided inequality
$A \lesssim B \lesssim A$.
\item 
Let $l \in \mathbb{N}$.
We denote by $M^{l}$ the $l$-fold composition of the Hardy-Littlewood maximal operator $M$.
We write $M^0$ for the operator
\[
f \in L^0(\mathbb{R}^n) \mapsto |f| \in L^0(\mathbb{R}^n).
\]
The symbol $M^l_{\rm loc}$ corresponds to the one for the local
Hardy-Littlewood maximal operator $M_{\rm loc}$.
\item 
Given a measurable set $E \subset \mathbb{R}^n$, the characteristic function of $E$ is denoted by $\chi_E$.
We write $|E|$ for the Lebesgue measure of $E$.

\item 
The space $L^\infty_{\mathrm{c}}(\mathbb{R}^n)$ denotes the set of bounded and compactly supported measurable functions on $\mathbb{R}^n$.

\item 
For a measurable set $E \subset \mathbb{R}^n$ and a weight $w$, the notation
\[
w(E) := \int_E w(x)\,\mathrm{d}x
\]
stands for the integral of $w$ over $E$.

\item 
A weight is a measurable function $w:\mathbb{R}^n \to [0,\infty]$ that is positive and finite almost everywhere.

A locally integrable weight $w$ belongs to the \textit{$A_1$-class} (or is an \textit{$A_1$-weight}) if $0 < w < \infty$ almost everywhere and
\[
[w]_{A_1} := \left\| \frac{M w}{w} \right\|_{L^\infty} < \infty.
\]
The local $A_1$ constant is defined similarly, with $M$ replaced by a localized maximal operator $M_{{\rm loc}}$:
\[
[w]_{A_{1,\text{loc}}} := \left\| \frac{M_{{\rm loc}} w}{w} \right\|_{L^\infty} < \infty.
\]

\item 
Let $1 < p < \infty$. A locally integrable weight $w$ belongs to the \textit{$A_p$-class} (or is an \textit{$A_p$-weight}) if $0 < w < \infty$ almost everywhere and
\begin{equation}\label{eq:241226-111}
[w]_{A_p} := \sup_{Q \in \mathcal{Q}} \left( \frac{1}{|Q|} \int_Q w(x)\,\mathrm{d}x \right) \left( \frac{1}{|Q|} \int_Q w(x)^{-\frac{1}{p-1}}\,\mathrm{d}x \right)^{p-1} < \infty,
\end{equation}
where $\mathcal{Q}$ is the collection of all cubes in $\mathbb{R}^n$ with sides parallel to the coordinate axes.

The local $A_p$ constant is defined by restricting the supremum in \eqref{eq:241226-111} to cubes with side length at most $1$:
\[
[w]_{A_{p,\rm loc}} := \sup_{\substack{Q \in \mathcal{Q} \\ \ell(Q) \leq 1}} \left( \frac{1}{|Q|} \int_Q w(x)\,\mathrm{d}x \right) \left( \frac{1}{|Q|} \int_Q w(x)^{-\frac{1}{p-1}}\,\mathrm{d}x \right)^{p-1}.
\]

\item 
We denote by $\|M\|_{X \to X}$ the operator norm of the 
Hardy-Littlewood maximal operator $M$ acting on a function space $X$. The operator norm $\|M_{{\rm loc}}\|_{X \to X}$ is defined analogously for the localized maximal operator $M_{{\rm loc}}$.

\item
Given two complex-valued functions $f, \, g$,
we formally write the $L^2$-inner product by
\[
\langle f,g \rangle:=\int_{\mathbb{R}^n}
f(x)\overline{g(x)} \, {\rm d}x .
\]
\end{itemize}

Here is the organization of this paper.
Section \ref{s3} is oriented to applications of 
Theorems \ref{thm:241102-1a}
and \ref{thm:241102-1ab}.
We obtain the boundedness properties of operators and the wavelet characterization
in Section \ref{s3}.
We also take up the vector-valued inequality of $M$ and the extension operator in Section \ref{s3}.
We will see that our method allows us to include more function spaces
than are dealt with in \cite{ZYY24}.
We take up various function spaces in Section \ref{s5}.
We compare what we can say from 
Theorems \ref{thm:241102-1a}
and \ref{thm:241102-1ab}
with known results
in Sections \ref{s3} and \ref{s5}.
Finally, Section \ref{s2} gives a short proof of Theorem \ref{thm:241102-1ab}.

\section{Main results: applications of the extrapolation}
\label{s3}

We present applications.
To this end,  we introduce some terminology.
We consider $X$-based Sobolev spaces over domains. 
First we recall the one over ${\mathbb R}^n$.
Let $s>0$. Recall that the Bessel potential $(1-\Delta)^{\frac{s}{2}}:={\mathcal F}^{-1}
[(1+|\cdot|^2)^{\frac{s}{2}}\mathcal{ F}f]$ is an 
isomorphism in ${\mathcal S}'({\mathbb R}^n)$; see \cite[p. 251]{Sawano2018}. 
\begin{definition}
Let $s>0$. Define the $X$-based Sobolev space 
$W^s_X({\mathbb R}^n)$ of order $s$ to be the set of all $f \in X({\mathbb R}^n)$ 
for which $(1-\Delta)^{\frac{s}{2}}f$ is represented by an element in $X$. 
Equip $W^s_X({\mathbb R}^n)$ with the norm given by 
$\|f\|_{W^s_X}:=\|(1-\Delta)^{\frac{s}{2}}f\|_X$ for all $f \in W^s_X({\mathbb R}^n)$. 
\end{definition}
Assume that $M$ is bounded on $X$ and on $X'$. Recall that the inverse operator 
$(1-\Delta)^{-\frac{s}{2}}$ of $(1-\Delta)^{\frac{s}{2}}$ has the convolution 
kernel $K$ satisfying $K \in L^1({\mathbb R}^n)$ (see \cite[Corollary 197]{book}). 
Hence the Bessel potential
$(1-\Delta)^{-\frac{s}{2}}$
is bounded on $X$ since we have a pointwise estimate 
$|(1-\Delta)^{-\frac{s}{2}}f(x)| \lesssim M f(x)$, $x \in {\mathbb R}^n$, 
for all $f \in X$; see \cite[Proposition 159]{book}. Thus, 
$W^s_X({\mathbb R}^n)$ is a Banach space.
Note that this is a natural extension of $L^{p,s}(w)$
considered in \cite{Izuki-Taiwan,INNS23}
if $X=L^p(w)$ with $1<p<\infty$ and $w \in A_p$.
By definition $W^0_X({\mathbb R}^n)=X$.
Furthermore, if $s$ is a positive integer,
by virtue of the boundedness of singular integral operators,
which is guaranteed by the boundedness of $M$ on $X$ and $X'$,
the space $W^s_X$ concides with the space of all
$f \in X$ such that any weak partial derivative up to order $s$ belong to $X$.
Something similar applies if we merely assume that
 $M_{\rm loc}$ is bounded on $X$ and on $X'$.
In this case we need to replace
$(1-\Delta)^{-\frac{s}{2}}$
with
$(1-t_0{}^2\Delta)^{-\frac{s}{2}}$
with some $0<t_0 \ll 1$.

Our definitions can be extended to domains.
Let $D$ be a bounded Lipschitz domain.
Let $k \in \mathbb{N}$, $1 < p < \infty$,
and $w \in A_p$.
Define $L^{p,k}(D,w)$ to be the set of all
$f \in L^p(w)$ for which $\partial^\alpha f$,
the weak partial derivative of order $\alpha$,
belongs to $L^p(w)$ for any multiindex
$\alpha$ with $|\alpha| \leq k$.

For a ball Banach function space $X$,
we write $X(D)$ to be the set of all
$f \in L^1_{\rm loc}(D)$ for which
$Z f$ belongs to $X$.
Let $Z$ be the zero extension operator.
The norm of $f \in X(D)$ is given by
$\|f\|_{X(D)} = \|Z f\|_X$.

We write $W^k_X(D)$ for the set of all
$f \in L^1_{\rm loc}(D)$ for which
$\partial^\alpha f$ belongs to $X(D)$
for any multiindex $\alpha$ with
$|\alpha| \leq k$.

Here is the organization of Section \ref{s3}.
Section \ref{s3.1a}
considers wavelet characterization.
Section \ref{s3.2a}
deals with vector-valued inequalities of the Hardy--Littlewood maximal operator.
We refine a recent result in \cite{ZYY24}
in
Section \ref{s3.3a}.

\subsection{Wavelet characterization}
\label{s3.1a}

Based on the fundamental wavelet theory
(see \cite{HW-book,Meyer-book} for example), 
we can
choose
compactly supported $C^K$-functions
\begin{equation}\label{eq:210511-1}
\varphi
\mbox{ and }
\psi^l \quad (l=1,2, \ldots , 2^n-1)
\end{equation}
so that the following conditions are satisfied:
\begin{enumerate}
\item[$(1)$]
For any $J\in \mathbb{Z}$, 
the system
\[
\left\{ \varphi_{J,k}, \, \psi^l_{j,k} \, : \, 
k\in \mathbb{Z}^n , \, j\ge J, \, l=1,2, \ldots , 2^n-1 \right\}
\]
is an orthonormal basis of $L^2({\mathbb R}^n)$.
Here, given a 
function $F$ defined on $\mathbb{R}^n$, we write 
\[
F_{j,k}:= 2^{\frac{j n}{2}}F(2^j \cdot -k)
\]
for $j \in {\mathbb Z}$ and $k\in {\mathbb Z}^n$.
\item[$(2)$]
Denote by ${\mathcal P}_{0}^\perp$
the set of all integrable functions having $0$ integral.
We have
\begin{equation}\label{eq:241226-1}
\psi^l \in  {\mathcal P}_{0}^\perp \quad
(l=1,2, \ldots , 2^n-1).
\end{equation}
In addition, 
they are real-valued 
and compactly supported with
\begin{equation}\label{eq:200111-1}
\mathrm{supp}(\varphi)=\mathrm{supp}(\psi^l)=[0,2N-1]^n
\end{equation}
for some $N\in \mathbb{N}$.
\end{enumerate}

We also consider
cubes 
$Q_{j,k}$ and $Q^*_{j,k}$ defined by
\[
Q_{j,k}:=\prod_{m=1}^n
\left[ 2^{-j}k_m , 2^{-j}(k_m+1) \right]
, \quad
Q^*_{j,k}:=\prod_{m=1}^n
\left[ 2^{-j}k_m , 2^{-j}(k_m+2N-1) \right] ,
\]
where $N\in \mathbb{N}$ 
satisfies (\ref{eq:200111-1}), and
\begin{equation}\label{eq:211202-206}
\chi_{j,k} := 
2^{\frac{j n}{2}}\chi_{Q_{j,k}}
, \quad
\chi^*_{j,k} := 
2^{\frac{j n}{2}}\chi_{Q^*_{j,k}}  .
\end{equation}

We note that
the 
$L^2$-inner product $\langle f,\phi \rangle$
exists 
for the complex-valued functions 
$f \in L^1_{\rm loc}({\mathbb R}^n)$ and $\phi \in C_{\rm c}$.
Then, using
the $L^2$-inner product, 
we define two square functions
$V f$,
$W_sf$ by
\begin{align*}
Vf&:= V^{\varphi}f:= \left( \sum_{k\in \mathbb{Z}^n} \left| 
\langle f,\varphi_{J,k} \rangle \chi_{J,k} \right|^2 \right)^{\frac12},\\
W_{s}f&:= W_{s}^{\psi^l}f:= \left( 
\sum_{l=1}^{2^n-1}
\sum_{j=J}^\infty
\sum_{k\in \mathbb{Z}^n} \left| 2^{js}
\langle f,\psi_{j,k}^l \rangle \chi_{j,k} \right|^2 \right)^{\frac12}.
\end{align*}
Here, 
$s\ge 0$ is a non-negative real number
and $J$ is a fixed integer.
We abbreviate $Wf:=W_0f$.

The wavelet characterization of $X$ can be obtained with ease.
\begin{theorem}\label{thm:241102-2}
Let $K>s>0$.
Suppose that we have $\varphi,\psi^l \in C^K({\mathbb R}^n)$
as in $(\ref{eq:210511-1})$.
Assume that
$(\ref{eq:241226-1})$ holds,
that 
$M_{\rm loc}$ is bounded on $X$ and on $X'$.
Then the following are equivalent for any $f \in X${\rm:}
\begin{enumerate}
\item[$(1)$]
$V f+W_s f \in X$.
\item[$(2)$]
$f \in W^s_X({\mathbb R}^n)$.
\end{enumerate}
Furthermore, in this case
$\|f\|_{W^s_X} \sim \|V f+W_s f\|_X$.
\end{theorem}

\begin{proof}
Let $0<t_0\ll1$.
Then we have
\begin{equation}\label{eq:250102-1}
((1-t_0{}^2\Delta)^{\frac{s}{2}}f,V f+W_s f) \in {\mathcal F}_{\rm loc}
\end{equation}
and
\begin{equation}\label{eq:250102-2}
(V f+W_s f,(1-t_0{}^2\Delta)^{\frac{s}{2}}f) \in {\mathcal F}_{\rm loc}
\end{equation}
from \cite[Theorem 4.6]{INNS23}.
\end{proof}

Despite the above observations, 
it is helpful to offer some words
(\ref{eq:250102-1})
and
(\ref{eq:250102-2})
when $s$ is an integer.
\begin{remark}
Let $X = L^p(\mathbb{R}^n)$. Relations
\eqref{eq:250102-1} and \eqref{eq:250102-2} are essentially proved in the books \cite{HW-book, Meyer-book}.

Now let $X = L^p(w)$ with $w \in A_p$. In this weighted setting, Lemari\'e-Rieusset \cite{Le} has established the result for the case $s = 0$.

Further, consider $X = W^{s,p}(w)$ with $w \in A_p$ and $s \in \{0,1,2,\ldots\}$. The first author \cite{Izuki-Taiwan} has extended the result to this broader context.

Therefore, we are in a position to apply Theorems \ref{thm:241102-1a} and \ref{thm:241102-1ab} once again.
\end{remark}

Letting \( s = 0 \), we derive the following corollary. 
It suffices to assume \emph{a priori} that \( f \) is locally integrable.

\begin{corollary}\label{thm:241102-1}
Assume that
$M_{\rm loc}$ is bounded on $X$ and on $X'$,
or equivalently,
each ${\mathcal R}_j$ is bounded on $X$.
Then the following are equivalent for any $f \in L^1_{\rm loc}({\mathbb R}^n)${\rm:}
\begin{enumerate}
\item[$(1)$]
$V f+W f \in X$.
\item[$(2)$]
$f \in X$.
\end{enumerate}
Furthermore, in this case
$\|f\|_X \sim \|V f+W f\|_X$.
\end{corollary}

The proof is the same as Theorem \ref{thm:241102-2},
although we assume that $f$ is merely locally integrable.

It is interesting to compare Corollary \ref{thm:241102-1}
with the following result:
\begin{theorem}
\label{thm:241226-1}
Assume that
$M_{\rm loc}$ is bounded on $X$ and on $X'$
and that
$X$ is separable.
Let $f \in X$.
Then
\begin{equation}\label{eq:250114-3}
f=\sum_{k \in {\mathbb Z}^n}
\langle f,\varphi_{J,k} \rangle \varphi_{J,k}
+
\sum_{l=1}^{2^n-1}
\sum_{j=J}^\infty
\sum_{k \in {\mathbb Z}^n}
\langle f,\psi^l_{j,k} \rangle \psi^l_{j,k}
\end{equation}
holds unconditionally in the topology of $X$.
\end{theorem}
See \cite[Theorem 8]{Karlovich21} for the case of $n=1$.
Once we use Corollary \ref{thm:241102-1}, Theorem \ref{thm:241226-1} is easy to prove. In fact, since \( X \) is separable, it is straightforward to show that \( X \) has the absolutely continuous norm. Corollary \ref{thm:241102-1} ensures that equation (\ref{eq:250114-3}) converges in \( X \).

Let ${\mathcal B}$ denote the set of all Borel measurable sets. We define the set
\[
S_0 = \left\{ \sum_{j=1}^N a_j \chi_{E_j} : \{a_1, a_2, \ldots, a_N\} \subset {\mathbb C}, \{E_1, E_2, \ldots, E_N\} \subset {\mathcal B} \right\}.
\]
Assume that \( M \) is bounded in both \( X \) and \( X' \).

It is important to mention the approach taken by Karlovich, who refined the results in \cite[Theorem 4.1]{INS-2015}, assuming that \( X \) is separable as we have mentioned. If either $X$ is separable or if \( L^2({\mathbb R}^n) \cap X \) is dense in \( X \), then Karlovich \cite{Karlovich21} and the authors in \cite{INS-2015} proved unconditional convergence in (\ref{eq:250114-3}), respectively.
Among others, Karlovich used the fact that $S_0 \cap X$ is dense in $X$
assuming that $X$ is separable.
It is noteworthy that $X$ has absolutely continuous norm if $X$ is separable.

Furthermore, in any of these cases the set
\[
{\mathcal X} = \left\{ \psi^l_{j,k} : l = 1, 2, \ldots, 2^n - 1, j \ge J, k \in {\mathbb Z}^n \right\} \cup \left\{ \varphi_{J,k} : k \in {\mathbb Z}^n \right\}
\]
is dense in \( X \). In particular, this implies that \( {\mathcal X} \) is dense in \( X \).
It is clear that $X$ is separable
and that $L^2({\mathbb R}^n) \cap X$ are contained
in ${\mathcal X}$.
This means that the assumptions
in \cite[Theorem 8]{Karlovich21} and \cite[Theorem 4.1]{INS-2015}
are equivalent.

\subsection{Vector-valued maximal inequality}
\label{s3.2a}

A similar technique can be used to obtain the vector-valued inequality.
\begin{example}
Let $r\in(1,\infty]$.
Assume that
$M_{\rm loc}$ is bounded on $X$ and on $X'$.
Then there exists a constant $C>0$ such that
\begin{equation}\label{eq:250418-10}
\left\|\left(\sum_{j=1}^\infty M_{\rm loc}f_j{}^r\right)^{\frac1r}\right\|_X
\le C
\left\|\left(\sum_{j=1}^\infty|f_j|^r\right)^{\frac1r}\right\|_X
\end{equation}
for all $\{f_j\}_{j=1}^\infty \subset X$
since
\[
\left(\left(\sum_{j=1}^\infty M_{\rm loc}f_j{}^r\right)^{\frac1r},\left(\sum_{j=1}^\infty|f_j|^r\right)^{\frac1r}\right)
\in {\mathcal F}_{\rm loc}
\]
for $N(t)=\alpha t^\beta$ where $t \gg 1$ and $\alpha$, $\beta$ are constants.
This result recaptures \cite[Theorem 3.1]{CMM22}. See
\cite[Theorem 303]{book}
for the case of $M$ to find
\[
\left\|\left(\sum_{j=1}^\infty M f_j{}^r\right)^{\frac1r}\right\|_{L^p(w)}
\le \alpha([w]_{A_p})^\beta
\left\|\left(\sum_{j=1}^\infty|f_j|^r\right)^{\frac1r}\right\|_{L^p(w)}
\]
for all $\{f_j\}_{j=1}^\infty \subset X$.
Estimate (\ref{eq:250418-10}) corresponding to the vector-valued maximal inequality
for $M$ obtained in \cite[Theorem 2.1]{MaSa25}.
\end{example}

\subsection{Extension operator}
\label{s3.3a}

Let $D$ be a bounded Lipshitz domain.
Chua established that there exists a bounded linear operator
$\Lambda:L^{p,k}(D,w) \to L^{p,k}(w)$
independent of $w$ such that
$\Lambda f|D=f$ for all $f \in L^{p,k}(D,w)$.
More precisely, letting $Z$ the zero extension operator,
Chua essentially proved that
\[
\left(\Lambda f,\sum_{|\alpha| \le k}|Z \partial^\alpha f|\right) \in {\mathcal F}
\]
with some suitable choice of $N(\cdot)$.
See \cite[Theorem 1.1]{Chua}.
Based on this fact, we investigate the action of the extension operator $\Lambda$
on $X$-based Sobolev spaces over $D$.
\begin{example}\label{example:250110-1}
Assume that
$M$ is bounded on $X$ and on $X'$,
or equivalently,
each ${\mathcal R}_j$ is bounded on $X$.
Let
$f \in W^k_X(D)$.
Then
$f \in L^{p,k}(w,D)$
for some $w \in A_1$.
Thus,
$\Lambda f$ makes sense
and
$\displaystyle \|\Lambda f\|_X \sim \sum_{|\alpha| \le k}\|Z\partial^\alpha f\|_X$.
\end{example}

We compare Example \ref{example:250110-1} with
an existing result \cite[Theorem 5.4]{ZYY24}.
\begin{remark}
\
\begin{enumerate}
\item
A recent work \cite[Theorem 5.4]{ZYY24} requires that 
$X^{\frac{1}{p}}$ is a ball Banach space for some $p > 1$
and that $M$ is bounded 
on $(X^{\frac{1}{p}})'$. This assumption implies that 
$M$ is bounded on both $X$ and $X'$, since it implies 
that ${\mathcal{R}}_j$ is bounded on $X$.
\item
A recent work \cite[Theorem 5.4]{ZYY24} requires that 
$X$ has the absolutely continuous norm. 
According to 
Theorems \ref{thm:241102-1a}
and \ref{thm:241102-1ab}, 
$X$ 
need not have the absolutely continuous norm.
\end{enumerate}
\end{remark}

\section{Examples of Banach lattices 
to which 
Theorems \ref{thm:241102-1a}
and \ref{thm:241102-1ab} are applicable}
\label{s5}

We will exhibit examples
 of Banach lattices 
to which 
Theorems \ref{thm:241102-1a}
and \ref{thm:241102-1ab}
are
applicable.
We concentrate on wavelet characterization.
However, it is worth mentioning that we have a counterpart
to vector-valued inequalities.

\subsection{Weighted Lebesgue spaces}


Let $1<q<\infty$ and
$w \in A_q$.
It is well known that $M$ and ${\mathcal R}_j$,
$j=1,2,\ldots,n$, are bounded on $L^q(w)$
\cite[Theorems 290 and 306]{book}.
Thus,
Theorems \ref{thm:241102-1a}
and \ref{thm:241102-1ab}
are applicable to $L^q(w)$.
If we merely assume that $w \in A_{q,\rm loc}$,
then only
Theorem \ref{thm:241102-1ab}
is applicable to $L^q(w)$.

Consider a more concrete case. Let 
$1<q<\infty$ and
$v(x)=|x|^\alpha$ be a  
power weight of order $\alpha \in {\mathbb R}$. According  
to \cite[Theorem 283]{book}, $v \in A_q$ if and only if  
$-n < \alpha < n(q-1)$. This condition also applies to the  
weight $W$, defined by $W(x)=\max(1,|x|)^\alpha$ for  
$x \in {\mathbb R}^n$. Therefore, if $-n < \alpha < n(q-1)$,  
Theorems \ref{thm:241102-1a}
and \ref{thm:241102-1ab} applicable to both $L^q(v)$  
and $L^q(W)$.

\subsection{Lorentz spaces}

Let 
$1 \le p<\infty$ and
$1 \le q \le \infty$.
Then the Lorentz space $L^{p,q}({\mathbb R}^n)$ is 
the set of all
$f \in L^0({\mathbb R}^n)$ for which the quasi-norm
\[
\|f\|_{L^{p,q}}
:=
\left\{
\int_0^{\infty} (t^{\frac{1}{p}}f^*(t))^{q}\frac{{\rm d}t}{t}
\right\}^{\frac{1}q}
\]
is finite.
Here $f^*$ stands for the decreasing rearrangement of $f$.
Remark that
if $1<p<\infty$,
then there exists a norm
$\|\cdot\|$ on $L^{p,q}({\mathbb R}^n)$
such that
\[
\|f\|_{L^{p,q}}
\sim
\|f\|
\]
for all $f \in L^{p,q}({\mathbb R}^n)$
\cite[Corollary 39]{book}.

Furthermore, since the space $L^{p,q}({\mathbb R}^n)$ can be
obtained through the real interpolation of the spaces
$L^{p_0}({\mathbb R}^n)$ and $L^{p_1}({\mathbb R}^n)$ as
shown in \cite{BeLo76}, the operators $M$ and
${\mathcal R}_j$ (for $j = 1, 2, \ldots, n$) are bounded on
$L^{p,q}({\mathbb R}^n)$ for $1 < p < \infty$ and $1 \le q
\le \infty$. Consequently, 
Theorems \ref{thm:241102-1a}
and \ref{thm:241102-1ab} can
be applied to $L^{p,q}({\mathbb R}^n)$.

To our knowledge, the wavelet characterization of the Lorentz
spaces $L^{p,q}({\mathbb R}^n)$ for $1 < p < \infty$ and $1
\le q \le \infty$ has been established by Soardi
\cite{Soardi}.

It is important to note that Lorentz spaces $L^{p,\infty}({
\mathbb R}^n)$, also known as the weak Lebesgue spaces
${\rm W}L^p({\mathbb R}^n)$, are not separable. Therefore,
in this case, we cannot apply Theorem \ref{thm:241226-1} to
obtain the unconditional convergence. However, according to Remark
\ref{rem:241226-1}, for any $\eta \in (1,p)$ and any $f \in
{\rm W}L^p({\mathbb R}^n)$, there exists an $A_1$-weight $w$
(which is independent of $f$) such that the convergence of
the expansion in Theorem \ref{thm:241226-1} occurs in
$L^{\eta}(w)$.

\subsection{Herz spaces}

For each
$k \in {\mathbb Z}$,
we set
${C}_k:=B(2^k) \setminus B(2^{k-1})$.
Let
$p,q \in [1,\infty]$ and $\alpha\in \mathbb{R}$. 
We write
$\chi_k=\chi_{C_k}$
for each $k \in {\mathbb Z}$.
The non-homogeneous Herz space $K_{p}^{\alpha,q}({\mathbb R}^n)$ consists of all 
$f \in L^0({\mathbb R}^n)$ for which
\[
\| f\|_{K_{p}^{\alpha,q}} := 
\| f\|_{L^{p}(B(1)) } 
+ 
\left(\sum_{k=1}^\infty(2^{k\alpha} \|f\chi_k\|_{L^p})^q\right)^{\frac1q}
<\infty
\]
and
the homogeneous Herz space $\dot{K}_{p}^{\alpha,q}({\mathbb R}^n)$ consists of all 
$f \in L^0({\mathbb R}^n)$ for which
\[
\| f\|_{\dot{K}_{p}^{\alpha,q}} := 
\left(\sum_{k=-\infty}^\infty(2^{k\alpha} \|f\chi_k\|_{L^p})^q\right)^{\frac1q}
<\infty. 
\]

Let 
$\alpha \in {\mathbb R}$,
$1
<p<\infty$
and
$1 \le q \le \infty$
satisfy
\begin{equation}\label{eq:230830-11a}
-\frac{n}{p}<\alpha<n-\frac{n}{p}.
\end{equation}
Then $M$ and ${\mathcal R}_j$ are bounded on $\dot{K}^{\alpha,q}_{p}({\mathbb R}^n)$ 
and $K^{\alpha ,q}_{p}({\mathbb R}^n)$
for all $j=1,2,\ldots,n$. See \cite{LY95,LY97,LY98}. Thus, 
Theorems \ref{thm:241102-1a}
and \ref{thm:241102-1ab}
are applicable to $K^{\alpha ,q}_{p}({\mathbb R}^n)$ 
and $\dot{K}^{\alpha ,q}_{p}({\mathbb R}^n)$. As far as we know, there are 
some results \cite{HWY,Tachi}
on wavelet characterization of Herz spaces. 
Since $M$ is bounded 
on $\dot{K}^{\alpha,q}_{p}({\mathbb R}^n)$, $K^{\alpha ,q}_{p}({\mathbb R}^n)$ 
and their preduals, it follows that many of these results fall within the scope 
of Corollary \ref{thm:241102-1}. 
Remark that $\dot{K}^{\alpha,\infty}_{p}({\mathbb R}^n)$ 
and $K^{\alpha,\infty}_{p}({\mathbb R}^n)$ are not separable. However, in view 
of Remark \ref{rem:241226-1}, we see that for any $\eta\in(1,p)$ and $f \in 
\dot{K}^{\alpha,\infty}_{p}({\mathbb R}^n) \cup K^{\alpha,\infty}_{p}({\mathbb R}^n)$ 
there exists an $A_1$-weight $w$ independent of $f$ such that the convergence 
of the expansion in Theorem \ref{thm:241226-1} takes place in $L^{\eta}(w)$.

\subsection{Weighted Lebesgue spaces with variable exponents}

Let $w$ be a weight.
The Lebesgue space $L^{p(\cdot)}(w)$ with variable exponent
consists of all
$f \in L^0({\mathbb R}^n)$ 
for which
\[
\| f\|_{L^{p(\cdot)}(w)} 
:= \inf
\left\{\lambda>0\,:\, \int_{\mathbb{R}^n} 
\left( \frac{|f(x)|}{\lambda} \right)^{p(x)}w(x){\rm d}x\le 1\right\}<\infty.
\]
We postulate some conditions
on variable exponents.
We use the following terminology:
we use the following standard notation
on exponents:
\begin{definition}\
\begin{enumerate}
\item
Let
$r(\cdot)$ be a variable exponent.
We write
\[
r_-:={\rm essinf}_{x \in {\mathbb R}^n}r(x), \quad
r_+:={\rm esssup}_{x \in {\mathbb R}^n}r(x).
\]
\item
The class
${\mathcal P}_0={\mathcal P}_0({\mathbb R}^n)$ collects all the variable exponents
$r(\cdot) \, : \, \mathbb{R}^n \to (0,\infty)$ that satisfy $0<r_-\le r_+<\infty$,
while
${\mathcal P}={\mathcal P}({\mathbb R}^n)$ collects all the variable exponents
$r(\cdot)$
in ${\mathcal P}_0$
that satisfy $r_->1$.
\end{enumerate}
\end{definition}
\begin{definition}
Let $r(\cdot) \in L^0({\mathbb R}^n)$.
\begin{enumerate}
\item
The variable exponent $r(\cdot)$ satisfies 
the local $\log$-H\"{o}lder continuity condition
if
\begin{equation}\label{logHolder}
|r(x)-r(y)|\le\frac{c_*}{\log(|x-y|^{-1})}
 \quad\text{for}\quad
x,y\in{\mathbb R}^n
\quad\mbox{with}\quad|x-y| \le \frac12.
\end{equation}
\item
If (\ref{logHolder}) is satisfied only for $y=0$,
then $r(\cdot)$ is $\log$-H\"{o}lder continuous at the origin.
The class
${\rm LH}_0={\rm LH}_0({\mathbb R}^n)$ stands for the class of exponents
which are $\log$-H\"{o}lder continuous at the origin.
\item
The exponent $r(\cdot)$ satisfies 
the $\log$-H\"{o}lder-type decay condition at infinity
if
\begin{equation}\label{decay}
 |r(x)-r_\infty|\le\frac{c^*}{\log(e+|x|)}
 \quad\text{for}\quad
 x\in{\mathbb R}^n.
\end{equation}
The class
${\rm LH}_\infty={\rm LH}_\infty({\mathbb R}^n)$ 
represents the class of exponents
that satisfy
$(\ref{decay})$.
\item
The class
${\rm LH}={\rm LH}({\mathbb R}^n)$ contains all exponents that satisfy $(\ref{logHolder})$
and $(\ref{decay})$.
\end{enumerate}
Here $c_*$, $c^*$ and $r_\infty$ are constants independent
of $x$ and $y$.
\end{definition}
For a variable exponent $r(\cdot)$ satisfying $1<r_- \le r_+<\infty$,
write
$r'(\cdot)=\frac{r(\cdot)}{r(\cdot)-1}$.
We employ the definition in \cite{CFN12}.
\begin{definition}
\label{definition-muckenhoupt-variable}
Suppose $p(\cdot) \in \mathcal{P}(\mathbb{R}^n)$. 
A weight $w$ is said to be an $A_{p(\cdot)}$-weight if $w$ satisfies
\[
\sup_{Q\in {\mathcal Q}} 
\frac{1}{|Q|} 
\left\| w^{\frac1{p(\cdot)}} \chi_Q \right\|_{L^{p(\cdot)}}
\left\| w^{-\frac1{p(\cdot)}} \chi_Q \right\|_{L^{p'(\cdot)}} 
< \infty.
\]
The Muckenhoupt class
$A_{p(\cdot)}$
adapted to $p(\cdot)$ consists of all $A_{p(\cdot)}$ weights. 
\end{definition}
Note that this definition extends the classical class
$A_p$. In fact, when $p(\cdot)$ is a constant function $p$,
the two definitions coincide.
It is known that
$M$ is bounded on $L^{p(\cdot)}(w)$
if
$p(\cdot) \in {\rm LH}$ satisfies
$1<p_- \le p_+<\infty$
and
$w \in A_{p(\cdot)}$.
Since this condition is symmetric;
in this case
$p'(\cdot) \in {\rm LH}$,
$1<p'_- \le p'_+<\infty$
and
$w^{-\frac{1}{p(\cdot)-1}} \in A_{p'(\cdot)}$,
Theorems \ref{thm:241102-1a}
and \ref{thm:241102-1ab}
are applicable to 
$L^{p(\cdot)}(w)$.
This recaptures a result by 
\cite{CFN12,CW,INS-2015}.

\subsection{Herz spaces with variable exponents}

Modelled on the previous section, 
it seems useful to study Herz spaces with variable exponents. 
The first author
\cite{IzukiAnal2010}
has established 
a method of wavelet characterization
applying
the boundeness of sublinear operators
and norm estimates for characteristic functions 
with variable exponent.
Recently more and more papers have appeared in this direction.
In this subsection
we establish a new method of wavelet characterization
as an application of the extrapolation theorem.

We now set up some notation to recall the definition 
of Herz spaces with three variable exponents
established by
the authors' previous work in
\cite{INS24}. 
Let
$p(\cdot)$,
$q(\cdot)$
be variable exponents.
The space
$\ell^{q(\cdot)}(L^{p(\cdot)}(w))$
is the set of all
sequences
$\{f_j\}_{j \in {\mathbb Z}}$
in $L^0({\mathbb R}^n)$
such that
\[
\|\{f_j\}_{j \in {\mathbb Z}}\|_{\ell^{q(\cdot)}(L^{p(\cdot)}(w))}
:=
\inf\left\{
\mu>0\,:\,
\rho_{\ell^{q(\cdot)}(L^{p(\cdot)}(w))}
(\{\mu^{-1}f_j\}_{j \in {\mathbb Z}})
\le 1
\right\}<\infty,
\]
where
\[
\rho_{\ell^{q(\cdot)}(L^{p(\cdot)}(w))}
(\{\mu^{-1}f_j\}_{j \in {\mathbb Z}})
:=
\sum_{j=-\infty}^\infty
\|\,|\mu^{-1}f_j|^{q(\cdot)}\,\|_{L^{p(\cdot)/q(\cdot)}(w)}.
\]
We sometimes consider the case where
${\mathbb Z}$ is replaced by ${\mathbb N}$.

Now, we recall the definition of two-weighted Herz spaces with three variable exponents. 
\begin{definition} 
Let $v$, 
$w$ be weights on $\mathbb{R}^n$, 
$p(\cdot),q(\cdot)\in{\mathcal P}_0(\mathbb{R}^n)$ and $\alpha(\cdot)\in L^{\infty}(\mathbb{R}^n)$. 
\begin{enumerate}
\item The non-homogeneous two-weighted Herz space $K_{p(\cdot)}^{\alpha(\cdot),q(\cdot)}(v,w)$ consists of all 
$f\in L^0(\mathbb{R}^n)$ for which
\[
\| f\|_{K_{p(\cdot)}^{\alpha(\cdot),q(\cdot)}(v,w)} := 
\| f\chi_{B_0}\|_{L^{p(\cdot)}(w) } 
+ \left\| \left\{ [v(B_k)]^{\alpha(\cdot)/n} f\chi_{C_k} \right\}_{k=1}^{\infty} \right\|_{\ell^{q(\cdot)}(L^{p(\cdot)}(w))}
\]
is finite.
\item 
The homogeneous two-weighted Herz space $\dot{K}_{p(\cdot)}^{\alpha(\cdot),q(\cdot)}(v,w)$ consists of all
$f\in L^0(\mathbb{R}^n)$ for which
\[
\| f\|_{\dot{K}_{p(\cdot)}^{\alpha(\cdot),q(\cdot)}(v,w)} := 
\left\| \left\{ [v(B_k)]^{\alpha(\cdot)/n} f\chi_k \right\}_{k=-\infty}^{\infty} \right\|_{\ell^{q(\cdot)}(L^{p(\cdot)}(w))}
<\infty. 
\]
\item
In the above, if $v=w=1$,
then we write
$\dot{K}_{p(\cdot)}^{\alpha(\cdot),q(\cdot)}({\mathbb R}^n)$
instead of
$\dot{K}_{p(\cdot)}^{\alpha(\cdot),q(\cdot)}(v,w)$
and
omit
$(v,w)$ in the norm
$\|\cdot\|_{\dot{K}_{p(\cdot)}^{\alpha(\cdot),q(\cdot)}(v,w)}$.
Analogously,
we define
$K_{p(\cdot)}^{\alpha(\cdot),q(\cdot)}({\mathbb R}^n)$.
The space
$\dot{K}_{p(\cdot)}^{\alpha(\cdot),q(\cdot)}({\mathbb R}^n)$
is called
homogeneous (non-weighted) Herz spaces with three variable exponents,
while
the space
$K_{p(\cdot)}^{\alpha(\cdot),q(\cdot)}({\mathbb R}^n)$
is called 
non-homogeneous (non-weighted) Herz spaces with three variable exponents.
\end{enumerate}
\end{definition}

Let $p(\cdot)\in{\mathcal P}(\mathbb{R}^n)\cap {\rm LH}(\mathbb{R}^n)$, 
$q(\cdot)\in{\mathcal P}_0(\mathbb{R}^n) \cap {\rm LH}_0({\mathbb R}^n) \cap {\rm LH}_\infty({\mathbb R}^n)$, 
$\alpha(\cdot)\in L^{\infty}(\mathbb{R}^n) \cap {\rm LH}_0({\mathbb R}^n) \cap {\rm LH}_\infty({\mathbb R}^n)$, 
$v\in A_{p_{v}}$ for some $p_{v}\in [1,\infty)$, $w\in A_{p_{(\cdot)}}$.
We know that there exist
constants $C>0$ and  $\delta \in (0,1)$
such that
\begin{align}
\frac{v(B)}{v(E)} & \le
C \left( \frac{|B|}{|E|} \right)^{p_v}, 
\label{lemma 5-21-1-i} \\
\frac{v(E)}{v(B)} & \le
C \left( \frac{|E|}{|B|} \right)^{\delta}
\label{lemma 5-21-1-ii} 
\end{align}
for all open balls $B$ and all measurable sets $E\subset B$.
See \cite{Duoa} for fundamental properties of the Muckenhoupt weights.
Among others, we refer to \cite[(7.3)]{Duoa} and \cite[Corollary 7.6]{Duoa}
for (\ref{lemma 5-21-1-i}) and (\ref{lemma 5-21-1-ii}),
respecitively.
Define
\begin{equation}\label{eq;230821-1}
v^-:=\begin{cases} \delta \ \ &\mbox{ if } \ \ \alpha_-\ge 0, \\ 
              p_{v}  \ \ &\mbox{ if } \ \ \alpha_-<0, 
\end{cases}\quad
v^+:=\begin{cases} p_{v} \ \ &\mbox{ if } \ \ \alpha_+\ge 0, \\
              \delta  \ \ &\mbox{ if } \ \ \alpha_+<0.
\end{cases}
\end{equation}
In our earlier paper
\cite{IzNo20},
the first and second authors established
that
there exist positive constants
$\delta_1=\delta_1(w,p(\cdot))$,
$\delta_2=\delta_2(w,p(\cdot))$,
$\delta_3=\delta_3(w,p(\cdot))$,
$\delta_4=\delta_4(w,p(\cdot))\in (0,1)$
and $C>0$ such that 
\begin{equation}
\frac{ \|\chi_{k} \|_{L^{p(\cdot)}(w)}}{\|\chi_{l} \|_{L^{p(\cdot)}(w)}}
\sim
\frac{ \|\chi_{k} \|_{ (L^{p'(\cdot)}(w^{-p(\cdot)/p'(\cdot)}))' }}
{\|\chi_{l} \|_{ (L^{p'(\cdot)}(w^{-p(\cdot)/p'(\cdot)}))' }}
 \le
C \left( \frac{|C_k|}{|C_l|} \right)^{\delta_1},
\label{3-13-5}
\end{equation}
that
\begin{equation}
\frac{ \|\chi_{k} \|_{(L^{p(\cdot)}(w))'}}{\|\chi_{l} \|_{(L^{p(\cdot)}(w))'}} \le
C \left( \frac{|C_k|}{|C_l|} \right)^{\delta_2} \label{3-13-6}
\end{equation}
that
\begin{equation}
\chi_{{\mathbb N} \times {\mathbb N}}(k,l)
\frac{ \|\chi_{k} \|_{L^{p(\cdot)}(w)}}{\|\chi_{l} \|_{L^{p(\cdot)}(w)}}
 \le
C \left( \frac{|C_k|}{|C_l|} \right)^{\delta_3},
\label{3-13-5a}
\end{equation}
and that
\begin{equation}
\chi_{{\mathbb N} \times {\mathbb N}}(k,l)\frac{ \|\chi_{k} \|_{(L^{p(\cdot)}(w))'}}{\|\chi_{l} \|_{(L^{p(\cdot)}(w))'}} 
\le
C \left( \frac{|C_k|}{|C_l|} \right)^{\delta_4} \label{3-13-6a}
\end{equation}
for all $k,l\in\mathbb{Z}$ with $k\le l$. 
Finally, assume that $\alpha(\cdot)$ satisfies
\begin{equation}\label{eq:230708-12}
-n\delta_1<v^-\alpha_-
\end{equation}
 and 
\begin{equation}\label{eq:230708-13}
v^+\alpha_+<n\delta_2.
\end{equation}
Then
in \cite{INS24}
the authors showed that there exists a constant $C>0$ such that
$$
\| {\mathcal R}_m f\|_{\dot{K}^{\alpha(\cdot), q(\cdot)}_{p(\cdot)}(v, w)}
\le
C \| f\|_{\dot{K}^{\alpha(\cdot), q(\cdot)}_{p(\cdot)}(v, w)} 
$$
for 
all $f\in L^\infty_{\rm c}({\mathbb R}^n)$
and
$m=1,2,\ldots,n$.
We have an analogy to the non-homogeneous space 
$K^{\alpha(\cdot), q(\cdot)}_{p(\cdot)}(v, w)$
if we replace
$(\ref{eq:230708-12})$
and
$(\ref{eq:230708-13})$
by
\begin{equation}\label{eq:230708-12k}
-n\delta_3<v^-\alpha_-
\end{equation}
 and 
\begin{equation}\label{eq:230708-13k}
v^+\alpha_+<n\delta_4.
\end{equation}

Hence,
Theorems \ref{thm:241102-1a}
and \ref{thm:241102-1ab}
are applicable to 
$\dot{K}^{\alpha(\cdot), q(\cdot)}_{p(\cdot)}(v, w)$
and
$K^{\alpha(\cdot), q(\cdot)}_{p(\cdot)}(v, w)$ in this case.

\subsection{Orlicz spaces}

First, let us recall the notion of Young functions.
A function $\Phi:[0,\infty) \to [0,\infty)$
is a {\it Young function},
if it satisfies the following conditions{\rm:}
\begin{enumerate}
\item[$(1)$]
\label{enum:190515-64}
$\Phi(0)=0$.
\item[$(2)$]
\label{enum:190515-65}
$\Phi$ is continuous.
\item[$(3)$]
\label{enum:190515-66}
$\Phi$ is convex.
That is,
$
\Phi((1-\theta)t_{1}+\theta t_{2})
\le
(1-\theta)\Phi(t_{1})+\theta\Phi(t_{2})
$
for all $t_{1},t_{2} \in [0,\infty)$ and $0<\theta<1$.
\end{enumerate}
Next, we recall the definition of Orlicz spaces
which
a Young function
$\Phi:[0,\infty) \to [0,\infty)$ 
generates.
We define
the {\it Luxemburg--Nakano norm}
$\|\cdot\|_{L^\Phi}$ by
\[
\| f \|_{L^\Phi}
:=
\inf\left(\left\{ \lambda \in(0,\infty) \, : \,
\int_{{\mathbb R}^n} \Phi\left(\frac{|f(x)|}{\lambda}\right) {\rm d}x 
\le 1 
\right\} \cup \{\infty\}\right)
\]
for $f \in L^0({\mathbb R}^n)$.
The {\it Orlicz space} $L^\Phi({\mathbb R}^n)$ over ${\mathbb R}^n$
is the set of
all $f \in L^0({\mathbb R}^n)$
for which $\| f \|_{L^\Phi}$ is finite.

In order that $M$ is bounded on both $L^\Phi({\mathbb R}^n)$ and on its K\"{o}the dual,
we recall the following classes.
\begin{enumerate}
\item[$(1)$]
\label{enum:190515-82}
A function $\varphi:(0,\infty) \to [0,\infty)$
or $\varphi:[0,\infty) \to [0,\infty)$
is a doubling function,
\index{doubling function@doubling function}
if there exists a constant $C>0$,
called a {\it doubling constant},
\index{doubling constant@doubling constant}
such that $C^{-1}\varphi(t) \le \varphi(s) \le C\varphi(t)$
for all $t$ and $s$ in the domain of $\varphi$
satisfying $s \le t \le 2s$.
\index{doubling function@doubling function}
\item[$(2)$]
\label{enum:190515-83}
Denote by $\Delta_{2}$ the set of all convex bijections
$\Phi:[0,\infty) \to [0,\infty)$
satisfying the doubling condition;
$\Phi(2r) \lesssim \Phi(r)$ for $r>0$.
The implicit constant is again called a
{\it doubling constant}.
In this case $\Phi$ also satisfies
the {\it $\Delta_{2}$-condition}.
\item[$(3)$]
\label{enum:190515-84}
The set $\nabla_{2}$ is the set of all convex bijections
$\Phi:[0,\infty) \to [0,\infty)$
such that there exists $C>1$ such that
$\Phi(2t) \ge 2C\Phi(t)$
for all $t \ge 0$.
In this case one says that $\Phi$ satisfies
the $\nabla_{2}$-{\it condition}.
\end{enumerate}
If $\Phi \in \Delta_2 \cap \nabla_2$,
then
$M$ is bounded on
 both $L^\Phi({\mathbb R}^n)$ and on its K\"{o}the dual.
 See
 \cite[Theorem 53]{book}
 and
\cite[Example 71]{book}.
Thus,
Theorems \ref{thm:241102-1a}
and \ref{thm:241102-1ab}
are applicable to 
$L^\Phi({\mathbb R}^n)$.
See \cite{Soardi} for wavelet characterization of Orlicz spaces.

\subsection{Morrey spaces}

Let $1\le r \le r_0 < \infty$.
For an $L^{r}_{\rm loc}({\mathbb R}^n)$-function $f$,
its (classical) Morrey norm is defined by
\begin{equation}\label{eq:130709-1A}
\| f \|_{{\mathcal M}^{r_0}_r}
:=
\sup_{x \in {\mathbb R}^n, R>0}
|B(x,R)|^{\frac{1}{r_0}-\frac{1}{r}}
\left(
\int_{B(x,R)}|f(y)|^{r}{\rm d}y
\right)^{\frac{1}{r}}.
\end{equation}
The Morrey space 
${\mathcal M}^{r_0}_r({\mathbb R}^n)$
is the set of all $f \in L^{r}_{\rm loc}({\mathbb R}^n)$
for which the norm $\| f \|_{{\mathcal M}^{r_0}_r}$ is finite.
This is a natural extension of $L^{r_0}({\mathbb R}^n)$
since ${\mathcal M}^{r_0}_{r_0}({\mathbb R}^n)$
and $L^{r_0}({\mathbb R}^n)$ are the same with coincidence of norms.

Let $1<r \le r_0<\infty$.
Then $M$ and ${\mathcal R}_j$, $j=1,2,\ldots,n$,
are bounded
on
${\mathcal M}^{r_0}_r({\mathbb R}^n)$.
See \cite{ChFr87}.
Thus,
Theorems \ref{thm:241102-1a}
and \ref{thm:241102-1ab}
are applicable to 
${\mathcal M}^{r_0}_r({\mathbb R}^n)$.

Let $1<r<r_0<\infty$.
Remark that the Morrey space ${\mathcal M}^{r_0}_r({\mathbb R}^n)$
is not separable but that
$M$ is bounded on 
${\mathcal M}^{r_0}_r({\mathbb R}^n)$ 
\cite[Theorem 382]{book}
and
${\mathcal M}^{r_0}_r({\mathbb R}^n)'$
\cite{SaTa09-2}. 
Thus,
in general
Morrey spaces do not fall within the scope of
Theorem \ref{thm:241226-1}.
Remark that
${\mathcal M}^{r_0}_r({\mathbb R}^n)$
fails to be reflexive \cite[Example 148]{book}.
Although
$L^{r_0}({\mathbb R}^n)$ is a subset
of ${\mathcal M}^{r_0}_r({\mathbb R}^n)$,
it fails to be dense \cite[Theorem 20]{book},
it is still possible to obtain the characterization
of Morrey spaces in terms of wavelets
as in \cite{Sawano08}.

\subsection{Besov--Bourgain--Morrey spaces}

Let $1 \le q \le p<\infty$ and $r,\tau\in [1,\infty]$.
The Besov-Bourgain--Morrey space 
$B{\mathcal M}^{p,\tau}_{q,r}({\mathbb R}^n)$
is defined to be the
set of all the
$f \in L^q_{\rm loc}({\mathbb R}^n)$
satisfying that
\begin{align*}
\|f\|_{B{\mathcal M}^{p,\tau}_{q,r}({\mathbb R}^n)}
:=
\left\|
\left\{
\left\|
\left\{
|Q_{\nu ,m}|^{\frac{1}{p}-\frac{1}{q}}
\left[\int_{Q_{\nu ,m}}|f(y)|^q{\rm d}y\right]^{\frac{1}{q}}
\right\}_{m \in {\mathbb Z}^n}
\right\|_{\ell^r({\mathbb Z}^n_m)}
\right\}_{\nu \in {\mathbb Z}}
\right\|_{\ell^\tau({\mathbb Z}_\nu)}<\infty.
\end{align*}
Here the inner $\ell^r$ norm is taken with respect to $m$
and the outer $\ell^\tau$ norm is taken with respect to $\nu$. 
See \cite{zyys24} for more.
The Bourgain--Morrey space ${\mathcal M}^p_{q,r}({\mathbb R}^n)$
is
the Besov--Bourgain--Morrey space $B{\mathcal M}^{p,r}_{q,r}({\mathbb R}^n)$
\cite{hnsh22}.
If
$1<q<p<r<\infty$
and
$1\le \tau<\infty$
or if
$1<q<p\le r<\tau=\infty$,
then $M$ and ${\mathcal R}_j$, $j=1,2,\ldots,n$,
are bounded
on
$B{\mathcal M}^{p,\tau}_{q,r}({\mathbb R}^n)$.
Thus,
Theorems \ref{thm:241102-1a}
and \ref{thm:241102-1ab}
are applicable to 
$B{\mathcal M}^{p,\tau}_{q,r}({\mathbb R}^n)$.

Remark that Besov--Bourgain--Morrey spaces fall within the scope of
Theorem \ref{thm:241226-1}
unlike
Morrey spaces.

\section{Appendix--Proof of Theorem \ref{thm:241102-1ab}}
\label{s2}

Let $\varepsilon > 0$ and $(f, g) \in \mathcal{F}_{{\rm loc}}$. It suffices to show that
\begin{equation}\label{eq:241101-1}
\|f\|_{X} \le C_1 \|g\|_X + C_2 \varepsilon \|f\|_X,
\end{equation}
where $C_1 > 0$ depends on $\varepsilon$, while $C_2 > 0$ is independent of $\varepsilon$ and the pair $(f, g)$.

In fact, for any $r > 0$, since $(\chi_{B(r)} \chi_{[0, r]}(|f|) f, g) \in \mathcal{F}_{{\rm loc}}$, we have
\[
\|\chi_{B(r)} \chi_{[0, r]}(|f|) f\|_{X} \le C_1 \|g\|_X + C_2 \varepsilon \|\chi_{B(r)} \chi_{[0, r]}(|f|) f\|_X
\]
once we establish \eqref{eq:241101-1}. Note that $\chi_{B(r)} \chi_{[0, r]}(|f|) f \in X$,
since $L^\infty_{\rm c}({\mathbb R}^n) \subset X$. 
By taking $\varepsilon = \frac{1}{2C_2}$, we obtain
\[
\|\chi_{B(r)} \chi_{[0, r]}(|f|) f\|_{X} \le 2C_1 \|g\|_X.
\]
Letting $r \to \infty$, we obtain the desired result.
Thus, we may assume that $f \in L^\infty_{\rm c}(\mathbb{R}^n)$.

We now prove \eqref{eq:241101-1}. Since $(f, g) \in \mathcal{F}_{{\rm loc}}$ if and only if $(|f|, |g|) \in \mathcal{F}_{{\rm loc}}$ for any $f, g \in L^0(\mathbb{R}^n)$, we may assume that $f, g$ are non-negative from now on. 

Furthermore, by adding a small constant multiple of 
\[
L:=\sum_{l=0}^\infty \frac{1}{(2\|M_{\mathrm{loc}}\|_{X' \to X'})^l}M_{\mathrm{loc}}^l\chi_{B(1)}
\] to $g$, we may assume that $g$ is positive everywhere. Finally, we may assume $g \in X$; otherwise, the conclusion is trivial.

\medskip

Define, for any measurable function $k$,
\[
R_k(x) := \sum_{l=0}^\infty \frac{M_{\mathrm{loc}}^l k(x)}{(\alpha \|M_{\mathrm{loc}}\|_{X \to X})^l}, \quad
R'_k(x) := \sum_{l=0}^\infty \frac{M_{\mathrm{loc}}^l k(x)}{(\alpha \|M_{\mathrm{loc}}\|_{X' \to X'})^l}, \quad x \in \mathbb{R}^n,
\]
where $\alpha \ge 2$ is a constant to be fixed later (see \eqref{eq:250114-2}).

We will use the inequality
\begin{equation}\label{eq:250105-1}
ab^{p-1} \le C_\varepsilon a^p + \varepsilon b^p \quad (a, b, \varepsilon > 0),
\end{equation}
where $C_\varepsilon > 0$ depends only on $\varepsilon$ and $p$.

To estimate $\|f\|_X$, fix a non-negative function $h \in X'$ with $\|h\|_{X'} < 1$ to dualize. It suffices to show that there exist constants $C_0$ (independent of $\varepsilon$, $f$, $g$, and $h$) and $C_\varepsilon$ (dependent on $\varepsilon$ but not on $f$, $g$, or $h$) such that
\begin{equation}\label{eq:250105-10}
\int_{\mathbb{R}^n} f(x) h(x){\rm d}x \le C_\varepsilon \|g\|_X + C_0 \varepsilon \|f\|_X.
\end{equation}
Indeed, since $X = (X')'$ as in \cite{book}
and $f \in L^\infty_{\rm c}(\mathbb{R}^n) \subset X$, this implies \eqref{eq:241101-1}.

We may also assume that $h$ is positive by adding a small constant multiple of $L$ if necessary. So from here on, assume $f \in L^\infty_{\rm c}(\mathbb{R}^n)$ and $g, h > 0$.

\medskip

Assuming $f \ge 0$, we observe that $R_{g+f} > 0$ and $R'_h > 0$. Moreover, we have
\[
[R_{g+f}]_{A_{1,\rm loc}} \le \alpha \|M_{\mathrm{loc}}\|_{X \to X}, \quad
[R'_h]_{A_{1,\rm loc}} \le \alpha \|M_{\mathrm{loc}}\|_{X' \to X'}.
\]
Hence,
\begin{equation}\label{eq:250114-1}
[R_{g+f}^{1-p} R'_h]_{A_{p,\rm loc}}
\le [R_{g+f}]_{A_{1,\rm loc}}^{p-1} [R_{g+f}^{1-p} R'_h]_{A_{1,\rm loc}}
\lesssim 1,
\end{equation}
with the implicit constant independent of $f$ and $g$.

Since $h \le R'_h$, we get
\begin{align}\label{eq:250105-11}
\int_{\mathbb{R}^n} f(x) h(x){\rm d}x
\le \int_{\mathbb{R}^n} f(x) R'_h(x){\rm d}x 
= \int_{\mathbb{R}^n} f(x) R_{g+f}(x)^{p-1} \cdot R_{g+f}(x)^{1-p} R'_h(x){\rm d}x.
\end{align}
Applying \eqref{eq:250105-1}, we obtain
\begin{align}\label{eq:250105-12}
&\int_{\mathbb{R}^n} f(x) R_{g+f}(x)^{p-1} R_{g+f}(x)^{1-p} R'_h(x){\rm d}x 
\nonumber\\
&\le C_\varepsilon \int_{\mathbb{R}^n} f(x)^p R_{g+f}(x)^{1-p} R'_h(x){\rm d}x
+ \varepsilon \int_{\mathbb{R}^n} R_{g+f}(x)^p R_{g+f}(x)^{1-p} R'_h(x){\rm d}x.
\end{align}

We estimate the first term. Since $(f, g) \in \mathcal{F}_{{\rm loc}}$, we have
\begin{align}\label{eq:250105-13}
\int_{\mathbb{R}^n} f(x)^p R_{g+f}(x)^{1-p} R'_h(x){\rm d}x
\nonumber
&\le N \left([R_{g+f}^{1-p} R'_h]_{A_{p,\rm loc}}\right)^p
\int_{\mathbb{R}^n} g(x)^p R_{g+f}(x)^{1-p} R'_h(x){\rm d}x \\
&\lesssim \int_{\mathbb{R}^n} g(x)^p R_{g+f}(x)^{1-p} R'_h(x){\rm d}x,
\end{align}
by \eqref{eq:250114-1}. Since $R_{g+f} \ge g$, it follows that
\begin{equation}\label{eq:250105-14}
\int_{\mathbb{R}^n} g(x)^p R_{g+f}(x)^{1-p} R'_h(x){\rm d}x
\le \int_{\mathbb{R}^n} g(x) R'_h(x){\rm d}x.
\end{equation}
Applying H$\mathrm{\ddot{o}}$lder's inequality and the boundedness of $R'_h$ on $X'$, we obtain
\begin{equation}\label{eq:250105-15}
\int_{\mathbb{R}^n} g(x) R'_h(x){\rm d}x \lesssim \|g\|_X \|R'_h\|_{X'} \lesssim \|g\|_X \|h\|_{X'} \lesssim \|g\|_X.
\end{equation}
Putting together \eqref{eq:250105-13}--\eqref{eq:250105-15}, we find
\begin{equation}\label{eq:250105-16}
\int_{\mathbb{R}^n} f(x)^p R_{g+f}(x)^{1-p} R'_h(x){\rm d}x \lesssim \|g\|_X.
\end{equation}

We handle the second term
 using the inequality
\[
\left(\sum_{l=1}^\infty a_l\right)^p \le \sum_{l=1}^\infty 2^{p l} a_l^p
\]
for non-negative sequences $\{a_l\}_{l=1}^\infty$. Applying this to $R_{g+f}$, we get
\begin{align}\label{eq:250105-21}
\int_{\mathbb{R}^n} R_{g+f}(x)R'_h(x){\rm d}x
&=
\int_{\mathbb{R}^n} R_{g+f}(x)^p R_{g+f}(x)^{1-p} R'_h(x){\rm d}x
\nonumber\\
&\le 2^p \sum_{l=0}^\infty \frac{2^{p l}}{(\alpha \|M_{\mathrm{loc}}\|_{X \to X})^{p l}}
\int_{\mathbb{R}^n} M_{\mathrm{loc}}^l[f+g](x)^p R_{g+f}(x)^{1-p} R'_h(x){\rm d}x.
\end{align}
We abbreviate
\[
W := R_{g+f}^{1-p} R'_h \in A_{p,\rm loc}.
\]
If we argue using the dyadic grids and using the idea of Lerner
\cite{Lerner08},
we find a constant $\beta > 1$ such that
\[
\int_{\mathbb{R}^n} M_{\mathrm{loc}} F(x)^p W(x) \, \mathrm{d}x 
\le \beta [W]_{A_{p,\rm loc}}^{p'} \int_{\mathbb{R}^n} |F(x)|^p W(x) \, \mathrm{d}x
\]
for all measurable functions $F$.
Applying this estimate iteratively, we obtain
\begin{align}\label{eq:250105-22}
\int_{\mathbb{R}^n} M_{\mathrm{loc}}^l [f + g](x)^p W(x) \, \mathrm{d}x 
\le (\beta [W]_{A_{p,\rm loc}}^{p'})^l \int_{\mathbb{R}^n} (f(x) + g(x))^p W(x) \, \mathrm{d}x.
\end{align}

Assume
\[
\alpha > \frac{2 \left( \beta [W]_{A_{p,\rm loc}}^{p'} \right)^{1/p}}{\|M_{\mathrm{loc}}\|_{X \to X}}.
\]
Then we have
\begin{equation}\label{eq:250114-2}
\sum_{l=0}^\infty 
\frac{2^{pl} \left( \beta [W]_{A_{p,\rm loc}}^{p'} \right)^l}{\left( \alpha \|M_{\mathrm{loc}}\|_{X \to X} \right)^{pl}} 
< \infty.
\end{equation}

Since $R_{g+f} \ge f + g$, it follows from \eqref{eq:250105-22} and related estimates that
\begin{align}\label{eq:250105-23}
\int_{\mathbb{R}^n} R_{g+f}(x)^p R_{g+f}(x)^{1-p} R'_h(x) \, \mathrm{d}x
&\lesssim \int_{\mathbb{R}^n} (f(x) + g(x)) R'_h(x) \, \mathrm{d}x 
\lesssim \|f\|_X + \|g\|_X.
\end{align}

Combining \eqref{eq:250105-12}, \eqref{eq:250105-16}, and \eqref{eq:250105-23}, we obtain \eqref{eq:250105-10}, and hence the desired estimate \eqref{eq:241101-1} is proved.
This completes the proof of Theorem \ref{thm:241102-1ab}.

\section*{Acknowledgement}

This work was partly supported by MEXT Promotion of Distinctive Joint Research Center Program JPMXP0723833165.
The authors are thankful to Professors Emiel Lorist
and David Cruz-Uribe for their remarks on Theorems \ref{thm:241102-1a}
and \ref{thm:241102-1ab} as well as the remark on weighted Sobolev spaces.
\section*{Declarations}

\begin{itemize}
\item Funding:
Yoshihiro Sawano was supported by Grant-in-Aid for Scientific Research (C) (19K03546), 
the Japan Society for the Promotion of Science.
\item Conflict of interest/Competing interests (check journal-specific guidelines for which heading to use)
Not applicable
\item Ethics approval and consent to participate:
Not applicable
\item Consent for publication:
Not applicable
\item Data availability:
Not applicable
\item Materials availability:
Not applicable
\item Code availability:
Not applicable
\item Author contribution:
The three authors contributed equally to the correctness of this article.
\end{itemize}


\begin{thebibliography}{999}
\bibitem{ABM} 
H. A. Aimar, A. L. Bernardis, and F. J. Mart\'{i}n-Reyes, 
Multiresolution approximations and wavelet bases of weighted $L^p$ spaces, 
\textit{J. Fourier Anal. Appl.}, \textbf{9} (2003), 497--510.

\bibitem{BeLo76} 
J. Bergh and J. L\"{o}fstr\"{o}m, Interpolation spaces. An introduction. 
Grundlehren der Mathematischen Wissenschaften, No. 223. Springer-Verlag, 
Berlin-New York, 1976.
\bibitem{BS}
C. Bennett and R. Sharpley, \textit{Interpolation of Operators}, Academic Press, 1988. ISBN: 9780123994585.

\bibitem{CMM22}
M. M. Cao, J. J. Mar\'{i}n, and J. M. Martell, Extrapolation on function and modular spaces, and applications, \textit{Adv. Math.}, \textbf{406} (2022), 108520. https://doi.org/10.1016/j.aim.2022.108520.

\bibitem{ChFr87} 
F. Chiarenza and M. Frasca, Morrey spaces and Hardy--Littlewood maximal function, \textit{Rend. Mat.}, \textbf{7} (1987), 273--279.

\bibitem{Chua}
S. K. Chua, Extension theorems on weighted Sobolev spaces, \textit{Indiana Univ. Math. J.}, \textbf{41} (1992), 1027--1076.

\bibitem{C17}
D. Cruz-Uribe, Extrapolation and Factorization.
Function Spaces, Embeddings and Extrapolation X, Paseky 2017. Edited by Lukes, J. and Pick, L. Prague: Matfyzpress, Charles University, 2017, 45--92. arXiv:1706.02620.

\bibitem{CFN12} 
D. Cruz-Uribe, A. Fiorenza, and C. J. Neugebauer, Weighted norm inequalities for the maximal operator on variable Lebesgue spaces, \textit{J. Math. Anal. Appl.}, \textbf{394} (2012), 744--760.

\bibitem{CW} 
D. Cruz-Uribe, SFO and L.-W. Wang, Extrapolation and weighted norm inequalities in the variable Lebesgue spaces, \textit{Trans. Amer. Math. Soc.}, \textbf{369} (2016), 1205--1235.

\bibitem{CGMP06} 
G. P. Curbera, J. Garc\'{i}a-Cuerva, J. M. Martell, and C. P\'{e}rez, Extrapolation with weights, rearrangement-invariant function spaces, modular inequalities, and applications to singular integrals, \textit{Adv. Math.}, \textbf{203} (2006), 256--318.

\bibitem{DSL23}
C. Deng, J. W. Sun, and B. Li, Extrapolations on ball Banach function spaces and applications, \textit{Ann. Funct. Anal.}, \textbf{14} (2023), no. 16. https://doi.org/10.1007/s43034-022-00236-y.



\bibitem{Duoa} 
J. Duoandikoetxea, Fourier Analysis, Graduate Studies in Mathematics, vol. 
29, American Mathematical Society, Providence, RI, 2001. Translated and 
revised from the 1995 Spanish original by D. Cruz-Uribe.

\bibitem{GM2001} 
J. Garc\'{i}a-Cuerva and J. M. Martell, Wavelet characterization of weighted 
spaces, \textit{J. Geom. Anal.}, \textbf{11} (2001), 241--264.

\bibitem{hnsh22} 
N. Hatano, T. Nogayama, Y. Sawano, and D. I. Hakim, Bourgain--Morrey spaces 
and their applications to boundedness of operators, \textit{J. Funct. Anal.}, 
\textbf{284} (2023), Paper No. 109720, 52 pp.

\bibitem{HW-book} 
E. Hern\'{a}ndez and G. Weiss, A First Course on Wavelets, CRC Press, Boca 
Raton, FL., 1996.

\bibitem{HWY} 
E. Hern\'{a}ndez, G. Weiss, and D. Yang, The $\varphi$-transform and wavelet 
characterizations of Herz-type spaces, \textit{Collect. Math.}, \textbf{47} 
(1996), 285--320.

\bibitem{ho12} 
K.-P. Ho, Generalized Boyds indices and applications, \textit{Analysis 
$($Munich$)$}, \textbf{32} (2012), 97--106.

\bibitem{IzukiGMJ} M. Izuki, Wavelets and modular inequalities in variable $L^p$ spaces, 
\textit{Georgian Math. J.}, \textbf{15} (2008), 281--293.

\bibitem{Izuki-Taiwan} M. Izuki, The characterization of weighted Sobolev spaces, 
\textit{Taiwanese J. Math.} {\bf 13} (2009), 467--492.

\bibitem{IzukiAnal2010}
M. Izuki,
Boundedness of sublinear operators on Herz spaces with variable exponent
and application to wavelet characterization,
\textit{Analysis Mathematica}, \textbf{36} (2010), 33--50.

\bibitem{INS-2015} M. Izuki, E. Nakai, and Y. Sawano, 
Wavelet characterization and modular inequalities for weighted Lebesgue spaces with variable exponent,
\textit{Annales Academia Scientiarum Fennica Mathematica}, \textbf{40} (2015), 551--571.

\bibitem{INNS23}
M. Izuki, T. Nogayama, T. Noi and Y. Sawano,
Wavelet characterization of local Muckenhoupt weighted Sobolev spaces with variable exponents, 
\textit{Constructive Approx.} \textbf{57} (2022), 161--234. 
https://doi.org/10.1007/s00365-022-09573-6

\bibitem{IzNo20}
M. Izuki and T. Noi,
Two weighted Herz spaces with variable exponents,
\textit{Bull. Malays. Math. Sci. Soc.},
\textbf{43} (2020), No.1, 169--200.

\bibitem{INS24} M. Izuki, T. Noi, and Y. Sawano, 
Extrapolation to two-weighted Herz spaces with three variable exponents, 
\textit{Advances in Operator Theory}, \textbf{9(2)} (2024), 1--30. 
https://doi.org/10.1007/s43036-024-00130-9


\bibitem{Tachi}
M. Izuki and K. Tachizawa,
Wavelet characterizations of weighted Herz spaces,
\textit{Sci. Math. Jpn.}, \textbf{67} (2008), 353--363.

\bibitem{Karlovich21} 
A. Yu. Karlovich, Wavelet bases in Banach function spaces, 
\textit{Bull. Malays. Math. Sci. Soc.}, \textbf{44} (2021), 3, 1669--1689.

\bibitem{Kop} 
T. S. Kopaliani, Greediness of the wavelet system in $L^{p(t)}(\mathbb{R})$, 
\textit{East J. Approx.}, \textbf{14} (2008), 59--67.

\bibitem{Lerner08}
A.K. Lerner, 
An elementary approach to several results
on the Hardy--Littlewood maximal operator.
\textit{Proc. Amer. Math. Soc.} {\bf 136} (2008), no. 8, 2829--2833.

\bibitem{Lerner07} 
A. K. Lerner and C. P\'{e}rez, A new characterization of the Muckenhoupt 
$A_p$ weights through an extension of the Lorentz--Shimogaki theorem, 
\textit{Indiana Univ. Math. J.}, \textbf{56} (2007), no. 6, 2697--2722.

\bibitem{Le} 
P. G. Lemari\'{e}-Rieusset, Ondelettes et poids de Muckenhoupt, 
\textit{Studia Math.}, \textbf{108} (1994), 127--147.

\bibitem{LY95} 
S. Lu and D. Yang, Oscillatory singular integrals on Hardy spaces 
associated with Herz spaces, \textit{Proc. Amer. Math. Soc.}, \textbf{123} 
(1995), 1695--1701.

\bibitem{LY97} 
S. Lu and D. Yang, Some characterizations of weighted Herz-type Hardy 
spaces and their applications, \textit{Acta Math. Sinica $($N.S.$)$}, 
\textbf{13} (1997), 45--58.

\bibitem{LY98} 
S. Lu and D. Yang, Boundedness of sublinear operators in Herz spaces on 
Vilenkin groups and its application, \textit{Math. Nachr.}, \textbf{191} 
(1998), 229--246.

\bibitem{MaSa25}
M. Masty{\l}o and Y. Sawano,
The Navier--Stokes equation in some general
Banach spaces, in preparation.

\bibitem{Meyer-book} 
Y. Meyer, Wavelets and Operators, Cambridge University Press, Cambridge, 1992.

\bibitem{Zoe}
Z. Nieraeth, 
Extrapolation in general quasi-Banach function spaces, 
\textit{J. Funct. Anal.} \textbf{285} (2023), no. 10, 110130. 
https://doi.org/10.1016/j.jfa.2023.110130

\bibitem{Nieraeth}
Z. Nieraeth,
The Muckenhoupt Condition. 2024. https://arxiv.org/pdf/2405.20907.



\bibitem{Rutsky} 
D. V. Rutsky, $A_1$-regularity and boundedness of Riesz transforms in Banach 
lattices of measurable functions, \textit{Zap. Nauchn. Sem. S.-Peterburg. 
Otdel. Mat. Inst. Steklov. $($POMI$)$}, \textbf{447} (2016), 113--122.

\bibitem{Sawano08}
Y. Sawano, 
Wavelet characterization of Besov-Morrey and Triebel-Lizorkin-Morrey spaces, 
Funct. Approx. Comment. Math., {\bf 38}, part 1, 93--107 (2008)

\bibitem{Sawano2018} 
Y. Sawano, Theory of Besov Spaces, Springer, 2018.

\bibitem{Sawano2024} 
Y. Sawano, Factorization of Hardy spaces using the ball Banach spaces, 
In: Conference Proceedings "The 50, 70, 80, $\ldots$ $\infty$ Conference in 
Mathematics" (Karlstad, Sweden, 2024), pp. 240-254, Element, Zagreb, 2024.

\bibitem{book} 
Y. Sawano, G. Di Fazio, and D. I. Hakim, Morrey Spaces. Vol. I. Introduction 
and applications to integral operators and PDE's, Monographs and Research 
Notes in Mathematics, CRC Press, Boca Raton, FL (2020), 479 pp. ISBN: 
978-1-4987-6551-0; 978-0-429-08592-5.

\bibitem{shyy17} 
Y. Sawano, K.-P. Ho, D. Yang and S. Yang, Hardy spaces for ball 
quasi-Banach function spaces. Dissertationes Math. (Rozprawy Mat.) \textbf{525}, 
1--102 (2017)

\bibitem{SaTa09-2} 
Y. Sawano and H. Tanaka, Predual spaces of Morrey spaces with non-doubling 
measures, Tokyo J. Math. \textbf{32} (2009), 471--486.

\bibitem{Soardi} 
P. Soardi, Wavelet bases in rearrangement invariant function spaces, 
\textit{Proc. Amer. Math. Soc.} \textbf{125} (1997), 3669--3673.

\bibitem{zyys24} 
Y. Zhao, J. Tao, D. Yang, W. Yuan, and Y. Sawano, Bourgain--Morrey Spaces 
mixed with structure of Besov spaces, \textit{Proc. Steklov Inst. Math.}, 
\textbf{323} (2024), 244--295.

\bibitem{ZYY24}
C.F. Zhu, D.C. Yang and W. Yuan, 
Bourgain-Brezis-Mironescu-type characterization of inhomogeneous ball Banach Sobolev spaces 
on extension domains, 
\textit{J. Geom. Anal.} \textbf{34} (2024), no. 10, Paper No. 295, 70 pp.


\end{thebibliography}
\end{document}